\documentclass[12pt,reqno]{amsart}
\usepackage{hyperref}
\usepackage{mathtools}
\usepackage{graphicx}
\usepackage{subcaption}
\usepackage{amsmath,mathtools,amsthm,amssymb,mathrsfs}

\DeclareMathAlphabet{\mathpzc}{OT1}{pzc}{m}{it}
\DeclareMathOperator{\RE}{Re}
\DeclareMathOperator{\IM}{Im}
\numberwithin{equation}{section}
\theoremstyle{plain}

\newtheorem{theorem}{Theorem}[section]
\newtheorem{lemma}[theorem]{Lemma}
\newtheorem{corollary}[theorem]{Corollary}
\newtheorem{remark}[theorem]{Remark}

\theoremstyle{definition}

\newtheorem{example}[theorem]{Example}

\allowdisplaybreaks
 
\begin{document}

\title[Combinations of Harmonic Mappings]{Directional Convexity of Combinations of Harmonic Half-Plane and Strip Mappings}

\author[S. Beig]{Subzar Beig}
\address{Department of Mathematics, Government Degree College Uri,
Baramulla--193 123,  Jammu \& Kashmir}
\email{beighsubzar@gmail.com}

\author[V. Ravichandran]{V. Ravichandran}
\address{Department of Mathematics, National Institute of Technology,
Tiruchirappalli -- 620 015, Tamil Nadu, India}
\email{ravic@nitt.edu, vravi68@gmail.com}

\dedicatory{Dedicated to the memory of Prof. Ataharul Islam}

\begin{abstract}
For $k=1,2$, let $f_k=h_k+\overline{g_k}$ be normalized harmonic right half-plane or vertical strip mappings. We consider  the convex combination $\hat{f}=\eta f_1+(1-\eta)f_2 =\eta h_1+(1-\eta)h_2 +\overline{\overline{\eta} g_1+(1-\overline{\eta})g_2}$   and   the combination $\tilde{f}=\eta h_1+(1-\eta)h_2+\overline{\eta g_1+(1-\eta)g_2}$. For real $\eta$, the two mappings $\hat{f}$ and $\tilde{f}$ are the same. We investigate the univalence and directional convexity of $\hat{f}$ and $\tilde{f}$ for $\eta\in\mathbb{C}$.  Some sufficient conditions are found for   convexity  of the combination $\tilde{f}$.
\end{abstract}

\keywords{harmonic mappings;   directional convexity;   harmonic shear; linear combination; strip mappings}

\subjclass[2020]{31A05; 30C45}
\maketitle

\section{Introduction}A domain $\Omega\subset \mathbb{C}$ is  convex in the direction  $\gamma$ $(0\leq\gamma<\pi)$, if every line parallel to the line joining the origin to the point $ e^{ i \gamma}$ has connected intersection with $\Omega$. For $\gamma=0$ (or $\pi/2)$,   a domain convex in the direction  $\gamma$  is said to be convex in the real (or imaginary) direction. A mapping $f $ is  convex in the  direction $\gamma$ if its image  is convex in the   direction $\gamma$. A mapping is convex if it is convex in every direction. Mappings  convex in some direction are    called as  the directionally  convex mappings. This paper studies the directional convexity of some combinations of harmonic mappings. Recall that a complex-valued harmonic function $f$ defined on the open unit disk $\mathbb{D}:=\left\lbrace z\in\mathbb{C}: |z|<1\right\rbrace$ can be written as $f=h+\overline{g}$, where the functions $h$ and $g$ are analytic   and are, respectively, known as analytic and co-analytic parts of $f$. By a theorem of  Lewy \cite{lewy}, it is follows that the function $f=h+\overline{g}$ is locally univalent and sense-preserving on $\mathbb{D}$ if, and only if its Jacobian $|h(z)'|^2-|g'(z)|^2>0$, or equivalently,  for $h'(z)\neq0$, the dilatation $\omega$ of $f$, defined by  $\omega=g'/h'$, satisfies $|\omega(z)|<1$ for all $z\in\mathbb{D}$.  Let $\mathcal{H}$ denotes the class of all locally univalent and sense-preserving harmonic mappings $f=h+\overline{g}$ defined on  $\mathbb{D}$ and normalized by the conditions $h(0)=h'(0)-1=0$. We shall be interested in the combinations of mappings in the subclass $\mathcal{S}_H$   of all univalent harmonic mappings in $\mathcal{H}$.

The  convex combination $\mathpzc{f}$ of the mappings  $f_k=h_k+\overline{g_k}$, $k=1,2$ in $\mathcal{S}_H$, given by \begin{equation}\label{p5eq1}
\mathpzc{f}=tf_1+(1-t)f_2=th_1+(1-t)h_2+\overline{tg_1+(1-t)g_2},\quad{} 0\leq t\leq1,
\end{equation}
is not univalent in general. See \cite{ahuja, subzar, subzar1, subzar2, subzar3, boyd, nowak} and the references therein for the other related work on the directional convexity of harmonic mappings and some of their combinations. Recently, several authors \cite{dorff, kumar, shi, sun, wang} have studied the convexity in a particular direction  of the convex combination of some subclasses of harmonic mappings  using the method of ``\textit{shear construction}'' \cite{cluine} which is described in the following lemma.

\begin{lemma} \cite{cluine}\label{p5lema6}
A locally univalent and sense-preserving harmonic mapping $f=h+\overline{g}$ on $\mathbb{D}$ is univalent and maps $\mathbb{D}$ onto a domain convex in the direction $\gamma$ $(0\leq \gamma<\pi)$ if and only if the analytic mapping $h-\emph{e}^{2 i \gamma}g$ is univalent and maps $\mathbb{D}$ onto a  domain convex in the direction $\gamma$.
\end{lemma}

Dorff and Rolf \cite{dorff} proved that the convex combination of two locally univalent sense-preserving harmonic mappings  is univalent and convex in the imaginary direction if they  are convex in the imaginary direction and have  the same dilatations.  Wang \textit{et al}.\ \cite{wang} proved that the mapping $\mathpzc{f}$  given by \eqref{p5eq1} is univalent and convex in the real direction if
\begin{equation}\label{p5eq2}
h_k(z)+g_k(z)=\frac{z}{1-z}.
\end{equation}
The results in \cite{wang} were extended to a larger class of mappings by Kumar \textit{et} \textit{al}.\ \cite{kumar}. Motivated by Wang \textit{et al}.\ \cite{wang} and Kumar \textit{et} \textit{al.}\  \cite{kumar}, we study the combinations of some harmonic mappings including the right half-plane and vertical strip mappings for directional convexity. For $\eta\in\mathbb{C}$ and $f_k=h_k+\overline{g_k}$ $(k=1,2)$  in ${\mathcal{S}_H}$, we define the mappings $\hat{f}$ and $\tilde{f}$ by
\begin{align}\label{np5eq5}
\hat{f}&=\eta f_1+(1-\eta)f_2=\eta h_1+(1-\eta)h_2+\overline{\overline{\eta} g_1+(1-\overline{\eta})g_2}
\intertext{and}
\tilde{f}&=\eta h_1+(1-\eta)h_2+\overline{\eta g_1+(1-\eta)g_2}\label{p5eq5}.
\end{align}
These mappings $\hat{f}$ and $\tilde{f}$ are same as the mapping $\mathpzc{f}$ defined in \eqref{p5eq1} when $0\leq \eta < 1$.

  It is well-known \cite{muhana, dorff1} that if the function $f=h+\overline{g}\in\mathcal{S}_H$ maps $\mathbb{D}$ onto the right half-plane $\{w\in\mathbb{C}: \RE(w)>-1/2\}$, then
\[h(z)+g(z)=\frac{z}{1-z}=\int_0^z\frac{d\xi}{(1-\xi)^2},\] and if it maps $\mathbb{D}$ onto the vertical strip $\{w\in\mathbb{C}:(\beta-\pi)/(2\sin\beta) <\RE w<\alpha/(2\sin\alpha) \}$, $\pi/2<\beta<\pi$, then
\[h(z)+g(z)=\frac{1}{2i\sin\beta}\log\left(\frac{1+z e^{i\beta}}{1+z e^{-i\beta}}\right)=\int_0^z\frac{d\xi}{1+2\xi \cos{\beta}+\xi^2}.\]%
In Section 2, we show that if the dilatation $|g_k'/h_k'| < \alpha_k\leq 1$ and \[h_k(z)+\emph{e}^{2 i \mu}g_k(z)=\int_0^z \psi_{\mu, \nu}(\xi)d\xi,\] where
\begin{equation}\label{np5eq6}
\psi_{\mu,\nu}(z)=\frac{1}{1-2z e^{ -i \mu}\cos\nu+z^2 e^{-2 i \mu}},\quad \mu\in[0, \pi),~ \nu\in[0, 2\pi),
\end{equation}
then the  mapping $\hat{f}$ is univalent and convex in the direction $\mu$ for all $\eta\in\mathbb{C}$ with \[ |\eta| <\frac{(1-\alpha_1)(1-\alpha_2)}{\alpha_1+\alpha_2}.\] The directional convexity of analytic mappings are verified by the following result of Royster and Ziegler.

\begin{lemma}\cite{royster}\label{p5theom7}
Let $\phi$ be a non-constant analytic mapping in $\mathbb{D}$. Then $\phi$ maps $\mathbb{D}$ onto a domain convex in the direction $\gamma$ $(0\leq \gamma<\pi)$ if, and only if, there are real numbers $\mu$ $(0\leq\mu<\pi)$ and $\nu$ $(0\leq\nu<2\pi)$, such that
\begin{equation}\label{p5eq8}
\RE\left( e^{ i (\mu-\gamma)}(1-2z e^{- i \mu}\cos\nu+z^2 e^{-2 i \mu})\phi'(z) \right)   \geq0,\quad z\in \mathbb{D}.
\end{equation}
\end{lemma}

\begin{remark}\label{p5remak7ab}
By taking $\gamma$ or $\gamma+\pi$ equal to $\mu$  in  Lemma \ref{p5theom7}, we see a non-constant analytic mapping $\phi$ is convex in the direction $(0\leq \mu< \pi)$, if for some $\nu$ $(0\leq \nu< 2\pi)$, $\RE\left(\phi'(z)/\psi_{\mu,\nu}(z)\right)$ is either non-negative or non-positive  on $\mathbb{D}$.
\end{remark}

 In Section 3, we show that if $|g_k'/h_k'| < \alpha_k\leq 1$ and $h_k-\emph{e}^{2 i \gamma}g_k=\psi$, where $\gamma\in[0, \pi)$ and $\psi$ is an analytic mapping convex in the direction $\gamma$, then the mapping $\tilde{f}$ is univalent and convex in the direction $\gamma$ for all $\eta\in\mathbb{C}$ with $|\eta| <(1-\alpha_1)(1-\alpha_2)/(\alpha_1+\alpha_2)$. However, if $\gamma=\mu+\pi/2$ and the function $\psi$ is replaced by the function $\int_0^z \psi_{\mu, \nu}(\xi)d\xi$ where the function $\psi_{\mu, \nu}$ is defined in \eqref{np5eq6}, then the mapping $\tilde{f}$ turns out to be convex.  Moreover, if $\gamma=\mu$ and the function $\psi$ is replaced by the function $\int_0^z p(\xi)\psi_{\mu, \nu}(\xi)d\xi$, where  $p$ is an analytic function with positive real part on $\mathbb{D}$, then the mapping $\tilde{f}$ is convex in the direction $\mu$. For specific choices of $p$, our results reduce to the results of  Wang \textit{et al.} \cite[Theorem 3]{wang} and Kumar \textit{et al.} \cite[Theorem 2.3]{kumar}.

\section{ The linear combination   $\hat{f}$ }
Our first  theorem gives us a condition on the parameter $\eta\in\mathbb{C}$ so that
the mapping $\hat{f}$ given by \eqref{np5eq5} is univalent and convex in the direction $\mu$.
\begin{theorem}\label{p5theom29s}
For $k=1,2$, let the  mapping $f_k=h_k+\overline{g_k}\in\mathcal{S}_H$ satisfy
\begin{equation}\label{np5eq1}
h_k(z)+\emph{e}^{2 i \mu}g_k(z)=\int_0^z \psi_{\mu, \nu}(\xi)d\xi,
\end{equation}
where the function $\psi_{\mu,\nu}$ is given by \eqref{np5eq6}. If the dilatation $\omega_k=g_k'/h_k'$ of $f_k$ satisfy the inequality $|\omega_k|<\alpha_k\leq 1$, then the mapping $\hat{f}$ given by \eqref{np5eq5} is univalent and convex in the direction $\mu$ for all $\eta\in\mathbb{C}$ with
\begin{equation}\label{np5eq3}
|\eta|\leq\alpha:=\frac{(1-\alpha_1)(1-\alpha_2)}{2(\alpha_1+\alpha_2)}.
\end{equation}
\end{theorem}

\begin{proof}We first  show that the mapping  $\hat{f}$ is locally univalent and sense-preserving. This is done by showing that the dilatation $\omega$ of the mapping $\hat{f}$ satisfies $|\omega|<1$ on $\mathbb{D}$. Since $\omega_k$ is the dilatation of the mapping $f_k$, the dilatation $\omega$ of the mapping  $\hat{f}$  is given by
\begin{align}
\omega&=\frac{\overline{\eta} g'_1+(1-\overline{\eta})g_2'}{\eta h'_1+(1-\eta)h_2'}=\frac{\overline{\eta}\omega_1h_1'+(1-\overline{\eta})\omega_2h_2'}{\eta h'_1+(1-\eta)h_2'}. \label{np5eq12}
\end{align}
Solving $g_k'=\omega_kh_k'$ along with  \eqref{np5eq1} for $h_k'$, we get
\[
h_k'=\frac{\psi_{\mu,\nu}}{1
+\emph{e}^{2 i \mu}\omega_k}.
\]
On using the above expression for $h_k'$, the equation \eqref{np5eq12} readily gives
\[
\omega=\frac{\overline{\eta}\omega_1(1+\emph{e}^{2 i \mu}\omega_2)+(1-\overline{\eta})\omega_2(1+\emph{e}^{2 i \mu}\omega_1)}{\eta(1+e^{2 i \mu}\omega_2)+(1-\eta)(1+\emph{e}^{2 i \mu}\omega_1)}.
\]
With $\omega_k$ replaced by $\emph{e}^{-2 i \mu}\omega_k$, the above equation gives
\[
\emph{e}^{2 i \mu}\omega=\frac{\overline{\eta}\omega_1(1+\omega_2)+(1-\overline{\eta})\omega_2(1+\omega_1)}{\eta(1+\omega_2)+(1-\eta)(1+\omega_1)}
\]
and thus  the dilatation $\omega$ satisfies $|\omega|<1$  on $\mathbb{D}$ if, and only if
\[|\overline{\eta}\omega_1(1+\omega_2)+(1-\overline{\eta})\omega_2(1+\omega_1)|^2<|\eta(1+\omega_2)+(1-\eta)(1+\omega_1)|^2,\]or equivalently if, and only if
\begin{equation}\label{np5eq2}
|1+\omega_1|^2\left(1-|\omega_2|^2\right)+2\RE\left(\overline{\eta}(\omega_2-\omega_1)(1+\overline{\omega_1})(\emph{e}^{2 i \theta}-\overline{\omega_2})\right)>0,
\end{equation}
where $\theta$ is the argument of $\eta$.
Therefore,  the dilatation $\omega$ satisfies $|\omega|<1$ on $\mathbb{D}$ if \[|\eta|<\frac{|1+\omega_1|\left(1-|\omega_2|^2\right)}{2|(\omega_2-\omega_1)(\emph{e}^{2 i \theta}-\overline{\omega_2})|}.\] Again, the inequality $|\omega_k|<\alpha_k$ implies that\[\frac{|1+\omega_1)|\left(1-|\omega_2|^2\right)}{2|(\omega_2-\omega_1)(\emph{e}^{2 i \theta}-\overline{\omega_2})|}>\frac{(1-\alpha_1)(1-\alpha_2)}{2(\alpha_1+\alpha_2)}=\alpha.\] Therefore, the dilatation $\omega$ of the mapping $\hat{f}$ satisfies  $|\omega|<1$ for all $\eta$ with $|\eta|\leq\alpha$ and, therefore, the mapping  $\hat{f}$ is locally univalent and sense-preserving.

We now show that  the mapping $h-\emph{e}^{2 i \mu}g$ is convex in the direction $\mu$ for all $\eta\in\mathbb{C}$ with $|\eta|\leq\alpha$.   As the mapping $\hat{f}$ is given by \eqref{np5eq5}, we have
\[ \hat{f}=\eta f_1+(1-\eta)f_2=:h+\overline{g}\]
where
\begin{align*}
h=\eta h_1+(1-\eta)h_2\quad \text{ and }\quad g=\overline{\eta}g_1+(1-\overline{\eta})g_2.
\end{align*}
Writing  $\eta=|\eta|e^{i\theta}$, we see that
\begin{align*}
h-e^{2 i \mu}g&=h_2-e^{2 i \mu}g_2+\eta(h_1-h_2-e^{2i( \mu-\theta)}(g_1-g_2))
\end{align*}
Therefore, in view of \eqref{np5eq1}, we see that
\begin{align}
\frac{h'-\emph{e}^{2 i \mu}g'}{\psi_{\mu,\nu}}&=\frac{h_2'-e^{2 i \mu}g_2'}{h_2'+e^{2 i \mu}g_2'}+\eta\left(\frac{h_1'-e^{2i( \mu-\theta)}g_1'}{h_1'+e^{2 i \mu}g_1'}-\frac{h_2'-e^{2i( \mu-\theta)}g_2'}{h_2'+e^{2 i \mu}g_2'}\right)\notag\\&=\frac{1-e^{2 i \mu}\omega_2}{1+e^{2 i \mu}\omega_2}+\eta\left(\frac{1-e^{2i( \mu-\theta)}\omega_1}{1+e^{2 i \mu}\omega_1}-\frac{1-e^{2i( \mu-\theta)}\omega_2}{1+e^{2 i \mu}\omega_2}\right)\notag\\&=\frac{(1-e^{2 i \mu}\omega_2)(1+e^{2 i \mu}\omega_1)+\eta e^{2 i \mu}(1+ e^{-2 i \theta})(\omega_2-\omega_1)}{(1+e^{2 i \mu}\omega_1)(1+e^{2 i \mu}\omega_2)}\notag\\&=\frac{\left(\splitfrac{(1-|\omega_2|^2-2i\IM(e^{2 i \mu}\omega_2))|1+e^{2 i \mu}\omega_1|^2}{+\eta e^{2 i \mu}(1+ e^{-2 i \theta})(\omega_2-\omega_1)(1+e^{-2 i \mu}\overline{\omega_1})(1+e^{-2 i \mu}\overline{\omega_2})}\right)}{|(1+e^{2 i \mu}\omega_1)(1+e^{2 i \mu}\omega_2)|^2}.\notag
\end{align}
Above equation shows that $\RE(h'-e^{2 i \mu}g')/\psi_{\mu,\nu}>0$  on $\mathbb{D}$ if, and only if
\[
(1-|\omega_2|^2)|1+ e^{2 i \mu}\omega_1|^2+\RE(\eta e^{2 i \mu}(1+ e^{-2 i \theta})(\omega_2-\omega_1)(1+e^{-2 i \mu}\overline{\omega_1})(1+e^{-2 i \mu}\overline{\omega_2}))>0.
\]
The last inequality holds  if
\begin{equation}
|1+e^{2 i \mu}\omega_1|^2\left(1-|\omega_2|^2\right)-2|\eta| |(\omega_2-\omega_1)(1+e^{-2 i \mu}\overline{\omega_1})(1+e^{-2 i \mu}\overline{\omega_2})|>0,
\end{equation}
or equivalently if
 \[|\eta|<\frac{|1+e^{2 i \mu}\omega_1|(1-|\omega_2|^2)}{2|(\omega_1-\omega_2)(1+e^{-2 i \mu}\overline{\omega_2})|} .\] But $|\omega_k|<\alpha_k$ implies that\[\frac{|1+e^{2 i \mu}\omega_1|(1-|\omega_2|^2)}{2|(\omega_1-\omega_2)(1+e^{-2 i \mu}\overline{\omega_2})|}>\frac{(1-\alpha_1)(1-\alpha_2)}{2(\alpha_1+\alpha_2)}=\alpha.\] Hence, it follow that  $\RE((h'-e^{2 i \mu}g')/\psi_{\mu,\nu})>0$ on $\mathbb{D}$ for all $\eta$ with $|\eta|\leq\alpha$. Therefore, by Remark \ref{p5remak7ab}, the mapping $h-\emph{e}^{2 i \mu}g$ is convex in the direction $\mu$.

Since  the mapping  $\hat{f}$ is locally univalent and sense-preserving and  the mapping $h-\emph{e}^{2 i \mu}g$ is   convex in the direction $\mu$, it follows by Lemma \ref{p5lema6} that the mapping $\hat{f}$ is univalent and convex in the direction~$\mu$ for all $\eta$ with $|\eta|\leq\alpha$.
\end{proof}
The following example gives an illustration of Theorem \ref{p5theom29s}.
\begin{example}
For $k=1,2$, let the mapping $f_k=h_k+\overline{g_k}$ be  such that
\begin{align*}
h_1(z)& =-\frac{5}{16}\left(-\frac{4z}{1-z}-\log(1-z)
        +\log\left(1-\frac{z}{5}\right)\right),\\
g_1(z)& =-\frac{5}{16}\left(\frac{4}{5}\frac{z}{1-z}+\log(1-z)
        -\log\left(1-\frac{z}{5}\right)\right),\\
h_2(z)& =\frac{7}{64}\left(\frac{8z}{1-z}-\log(1-z)
+\log\left(1+\frac{z}{7}\right)\right)
\intertext{and}
g_2(z)& =\frac{7}{64}\left(\frac{8}{7}\frac{z}{1-z}+\log(1-z)
-\log\left(1+\frac{z}{7}\right)\right).
\end{align*}
Then we have
\[ h_k(z)+g_k(z)=\int_0^z\frac{1}{(1-\xi)^2}\textit{d}\xi=\frac{z}{1-z},\]
 \[ \omega_1(z)=g_1'(z)/h_1'(z)=-z/5\quad \text{and}\quad \omega_2(z)=g_2'(z)/h_2'(z)=z/7.\] Hence, by Theorem 2.1, the mapping $\hat{f}=\eta f_1+(1-\eta)f_2$ is univalent and convex in the real direction for $\eta\in\mathbb{D}$.
 
\end{example}
\section{The  combination  $\tilde{f}$}
In this section, we find some sufficient conditions for the mapping $\tilde{f}$ defined by \eqref{p5eq5} to be univalent and convex in some direction. We examine separately the case when $\eta$ is real.

\begin{theorem}\label{p5theom11} Let $\psi$ be an analytic mapping convex in the direction $\gamma\in [0, \pi)$. For $k=1,2$, let $f_k=h_k+\overline{g_k}\in\mathcal{S_H}$ satisfy the condition
\begin{equation}\label{p5eq10}
\lambda(h_1-\emph{e}^{2 i \gamma}g_1)=h_2-\emph{e}^{2 i \gamma}g_2=\lambda\psi
\end{equation} for some $\lambda\in\mathbb{R}$. If any one of the following conditions holds:
\begin{itemize}
\item[(i)] $\lambda >0$ and  $0\leq \eta\leq1$, or
  $\lambda<0$ and $\eta\leq0$, or
\item[(ii)] $\lambda=1$,  the dilatation $\omega_k$ of $f_k$ satisfies $|\omega_k|<\alpha_k\leq 1$ and $\eta\in\mathbb{C}$ such that \[ |\eta|\leq\frac{(1-\alpha_1)(1-\alpha_2)}{2(\alpha_1+\alpha_2)},\]
\end{itemize}
then the mapping  $\tilde{f}$ given by \eqref{p5eq5} is univalent and convex in the direction $\gamma$.
\end{theorem}
\begin{proof}
Since \begin{align*} 
\tilde{f} & =\eta h_1+(1-\eta)h_2+\overline{\eta g_1+(1-\eta)g_2}=:h+\overline{g},
\intertext{the equation \eqref{p5eq10} shows that}
h-\emph{e}^{2 i \gamma}g
&=\eta\left(h_1-\emph{e}^{2 i \gamma}g_1-h_2+\emph{e}^{2 i \gamma}g_2\right)+h_2-\emph{e}^{2 i \gamma}g_2\\
&=\eta(\psi-\lambda\psi)+\lambda\psi
 =\left(\eta+\lambda\left(1-\eta\right)\right)\psi.
\end{align*} Therefore, in view of the assumptions on $\psi$ and $\lambda$, the mapping $h-\emph{e}^{2 i \gamma}g$ is convex in the direction $\gamma$.

Our result follows from Lemma \ref{p5lema6} if the mapping $\tilde{f}$ is locally univalent and sense-preserving.  We show this by proving  the dilatation $\omega$ of $\tilde{f}$  satisfies $|\omega |<1$.  Since $g_k'=\omega_kh_k'$, the dilatation $\omega$ of $\tilde{f}$ is given by
\begin{align}
\omega=\frac{g'}{h'}&=\frac{\eta g'_1+(1-\eta)g_2'}{\eta h'_1+(1-\eta)h_2'}=\frac{\eta\omega_1h_1'+(1-\eta)\omega_2h_2'}{\eta h'_1+(1-\eta)h_2'}. \label{p5eq12}
\end{align}
On using $g_k'=\omega_kh_k'$ in \eqref{p5eq10}, we see that
\begin{equation}\label{p5eq13}
h_1'=\frac{\psi'}{1-\emph{e}^{2 i \gamma}\omega_1}\quad{}\text{and}\quad{}h_2'=\frac{\lambda\psi'}{1-\emph{e}^{2 i \gamma}\omega_2}.
\end{equation}
Substituting the values of $h_1'$ and $h_2'$ from \eqref{p5eq13} in  \eqref{p5eq12}, we have
\begin{equation}\label{p5eq14}
\omega=\frac{\eta\omega_1(1-\emph{e}^{2 i \gamma}\omega_2)+\lambda(1-\eta)\omega_2(1-\emph{e}^{2 i \gamma}\omega_1)}{\eta(1-\emph{e}^{2 i \gamma}\omega_2)+\lambda(1-\eta)(1-\emph{e}^{2 i \gamma}\omega_1)}.
\end{equation}
With $\omega_k$ replaced by $\emph{e}^{-2 i \gamma}\omega_k$, the above reduced to
\begin{equation}\label{p5eq14a}
\emph{e}^{2 i \gamma}\omega=\frac{\eta\omega_1(1-\omega_2)+\lambda(1-\eta)\omega_2(1-\omega_1)}{\eta(1-\omega_2)+\lambda(1-\eta)(1-\omega_1)}.
\end{equation}

Case (i). If either $\eta$ is real with $0\leq \eta\leq1$ and $\lambda>0$, or $\eta$ is real with $\eta\leq0$ and $\lambda<0$, then both \[\frac{\eta}{\eta+\lambda(1-\eta)}\quad{} \text{ and} \quad{} \frac{\lambda(1-\eta)}{\eta+\lambda(1-\eta)}\] are non-negative, and at least one of them is positive.
In this case, it is easily seen that the denominator in the above expression of $\omega$ does not vanish in $\mathbb{D}$ for the values of $\eta$ and $\lambda$. Therefore, by using \eqref{p5eq14a}, it follows that
\begin{align}
\RE\left(\frac{1+\emph{e}^{2 i \gamma}\omega}{1-\emph{e}^{2 i \gamma}\omega}\right)&
=\RE\left(\frac{\eta(1+\omega_1)(1-\omega_2)+\lambda(1-\eta)(1+\omega_2)(1-\omega_1)}
{(\eta+\lambda(1-\eta))(1-\omega_2)(1-\omega_1)}\right)\notag\\
&=\RE\left(\frac{\eta}{\eta+\lambda(1-\eta)}\frac{1+\omega_1}{1-\omega_1}\right)
+\RE\left(\frac{\lambda(1-\eta)}{\eta+\lambda(1-\eta)}\frac{1+\omega_2}{1-\omega_2}\right).
\label{p5eq15}
\end{align}
 Since $|\omega_k|=|\emph{e}^{2 i \gamma}\omega_k|<1$, we have $\RE((1+\omega_k)/(1+\omega_k))>0$ on $\mathbb{D}$. Therefore, \eqref{p5eq15} shows that
  \[\RE\left(\frac{1+\emph{e}^{2 i \gamma}\omega}{1+\emph{e}^{2 i \gamma}\omega}\right)>0\] on $\mathbb{D}$. Hence $|\omega|=|\emph{e}^{2 i \gamma}\omega|<1$ on $\mathbb{D}$, which implies that $f$ is locally univalent and sense-preserving.

 Case (ii). For $\lambda=1$, we see from \eqref{p5eq14a} that
\[\emph{e}^{2i \gamma}\omega =\frac{\eta\omega_1(1-\omega_2)+(1-\eta)\omega_2(1-\omega_1)}
{\eta(1-\omega_2)+(1-\eta)(1-\omega_1)}.\]
Above equation shows that $|\omega|<1$ on $\mathbb{D}$ if, and only if
\[|\eta\omega_1(1-\omega_2)+(1-\eta)\omega_2(1-\omega_1)|^2<|\eta(1-\omega_2)+(1-\eta)(1-\omega_1)|^2,\]or equivalently if, and only if
\begin{equation}\label{p5eq16}
|1-\omega_1|^2\left(1-|\omega_2|^2\right)+2\RE\left(\eta(\omega_1-\omega_2)(1-\overline{\omega_1})(1-\overline{\omega_2})\right)>0.
\end{equation}
Therefore, $|\omega|<1$ on $\mathbb{D}$ if \[|\eta|<\frac{|1-\omega_1|\left(1-|\omega_2|^2\right)}{2|(\omega_1-\omega_2)(1-\omega_2)|}.\] But $|\omega_k|<\alpha_k$ implies that \[\frac{|1-\omega_1|\left(1-|\omega_2|^2\right)}{2|(\omega_1-\omega_2)(1-\overline{\omega_2})|}
>\frac{(1-\alpha_1)(1-\alpha_2)}{2(\alpha_1+\alpha_2)}.\] Hence, $|\omega|<1$ for all $\eta\in\mathbb{C}$ with \[|\eta|\leq\frac{(1-\alpha_1)(1-\alpha_2)}{2(\alpha_1+\alpha_2)}.\qedhere\]
\end{proof}

\begin{remark}
Since the mapping $\phi(z):=\int_0^z\psi_{\mu,\nu}(\xi)d\xi$, where $\psi_{\mu, \nu}$ is given by \eqref{np5eq6}, is convex (convexity of $\phi$ is easily seen by observing that $\RE\left(1+z\phi''/\phi'\right)>0$ on $\mathbb{D}$), and hence convex in the direction $\gamma$. Therefore, we can take $\psi=\psi_{\mu, \nu}$ in Theorem \ref{p5theom11}. However, in this case, we will show $\tilde{f}$ in Theorem \ref{p5theom11} belongs to class $\mathcal{K}_H$ of all convex harmonic  mappings in $\mathcal{S}_H$, provided $\gamma=\mu+\pi/2$ and $\lambda=1$. In fact, we have a more general result, see  Theorem \ref{p5theom22}.
\end{remark}

 For any non-negative integer $n$, define the differential operator $\mathcal{D}^n:A  \longrightarrow A$ on the class $A$ of all analytic mappings $f$ as: $
\mathcal{D}^0f(z)=f(z)$ and $\mathcal{D}^nf(z)=z(\mathcal{D}^{n-1}f)'(z)$ for $n\geq1$.
For the harmonic mapping $f=h+\overline{g}$, define $\mathcal{D}^nf:=\mathcal{D}^nh+\overline{\mathcal{D}^ng}$. In order to prove our next result, we use the following straight forward  generalization of Sheil-Small's \cite{sheil}  result on the relation between the starlike and convex harmonic mappings.
\begin{theorem}\label{p5theom21b}
If $f = h + \overline{g}$ is a starlike harmonic mapping in $\mathcal{S}_H$, and $H$ and $G$ are the analytic mappings defined by
\[\mathcal{D}^nH=h,\quad \mathcal{D}^nG=(-1)^ng,\quad H(0) = H'(0)-1=G(0) = 0,\]
then the mapping $F = H + \overline{G}\in\mathcal{K}_H$.
\end{theorem}

\begin{theorem}\label{p5theom22}
For $k=1,2$, $\mu\in[0, \pi)$ and  $\nu\in[0, 2\pi)$, let $f_k=h_k+\overline{g_k}$ be a harmonic mapping with $h(0)=h'(0)-1=0$. Let $\mathcal{D}^{n-1}f_k$ be locally univalent,  sense-preserving and
\begin{equation}\label{p5eq23}
\frac{h_k(z)+\emph{e}^{2 i \mu}(-1)^{n-1}g_k(z)}{z}=\frac{1}{z}\int_0^{z_n=z}\left(\cdots\frac{1}{z_1}\int_0^{z_1} \psi_{\mu,\nu}(\xi)\textit{d}\xi\cdots\right)\textit{d}z_{n-1},
\end{equation}
 where $\psi_{\mu, \nu}$ is given by \eqref{np5eq6}. If
\begin{itemize}
\item[(i)]  $0\leq \eta\leq1$, or
\item[(ii)] the dilatation $\omega_k$ of $\mathcal{D}^{n-1}f_k$ satisfies $|\omega_k|<\alpha_k\leq 1$   and  $\eta\in\mathbb{C}$ such that \[ |\eta|\leq\frac{(1-\alpha_1)(1-\alpha_2)}{2(\alpha_1+\alpha_2)},\]
\end{itemize} then the mapping $\tilde{f}$ given by \eqref{p5eq5} belongs to $\mathcal{K}_H$
\end{theorem}
\begin{proof}
Since
\begin{align}
\tilde{f}=\eta h_1+(1-\eta)h_2+\overline{\eta g_1+(1-\eta)g_2}=:h+\overline{g},\label{p5eq23a}
\end{align}
we have
\begin{align*}
h(z)+\emph{e}^{2 i \mu}g(z)&=\eta\left(h_1(z)+\emph{e}^{2 i \mu}g_1(z)-h_2(z)-\emph{e}^{2 i \mu}g_2(z)\right)+h_2(z)+\emph{e}^{2 i \mu}g_2(z)\\ &  =h_2(z)+\emph{e}^{2 i \mu}g_2(z).
\end{align*}
Let  $H(z):=\mathcal{D}^{n-1}h(z)$ and $G(z):=(-1)^{n-1}\mathcal{D}^{n-1}g(z)$.
In view of \eqref{p5eq23}, we see that
 \begin{equation}\label{p5eq24}H(z)+\emph{e}^{2 i \mu}G(z)
=\mathcal{D}^{n-1}h(z)+\emph{e}^{2 i \mu}(-1)^{n-1}\mathcal{D}^{n-1}g(z)=\int_0^z\psi_{\mu,\nu}(\xi)\textit{d}\xi,\end{equation} and hence $H'+ e^{-2 i \mu}G'=\psi_{\mu,\nu}$.
Theorem \ref{p5theom11}, in view of the assumptions on $\mathcal{D}^{n-1}f_k$, shows that the mapping $F:=H+\overline{G}$ is locally univalent and sense-preserving. We will show that it is convex. In view of Lemma \ref{p5lema6}, it suffices to show that the mapping $H- e^{2 i \theta}G$ is convex in the direction $\theta$ for all $\theta$ ranging in an interval of length $\pi$. In other words, it is sufficient to show that the mapping $ e^{ i (\mu-\theta)}(H- e^{2 i \theta}G)$ is convex in the direction $\mu$ for all $\theta$ such that $-\pi/2\leq\mu-\theta<\pi/2$. Since  $\tilde{f}$ is locally univalent and sense-preserving,  $|G'/H'|<1$ on $\mathbb{D}$, and hence  $$\RE\left(\frac{H'- e^{2 i \mu}G'}{H'+ e^{2 i \mu}G'}\right)>0.$$ Above inequality shows that
\begin{align} \RE \left(\frac{ e^{ i (\mu-\theta)}(H- e^{2 i \theta}G)'}{\psi_{\mu,\nu}}\right)  & =
\RE\left(\frac{ e^{ i (\mu-\theta)}(H- e^{2 i \theta}G)'}{H'+ e^{2 i \mu}G'}\right)\notag \\
&=\RE\left(\frac{( e^{ i (\mu-\theta)}H'- e^{2 i \mu} e^{- i (\mu-\theta)}G'}{H'+ e^{2 i \mu}G'}\right)\notag\\&=\RE\left(\cos(\mu-\theta)\frac{H'- e^{2 i \mu}G'}{H'+ e^{2 i \mu}G'}+ i \sin(\mu-\theta)\right)\notag\\&=\cos(\mu-\theta)\RE\left(\frac{H'- e^{2 i \mu}G'}{H'+ e^{2 i \mu}G'}\right)\geq 0.\label{p5eq24a}
\end{align}
 Therefore, in view of  \eqref{p5eq24a},   Remark \ref{p5remak7ab} shows that the mapping $ e^{ i (\mu-\theta)}$ $(H- e^{2 i \theta}G)$ is convex in the direction $\mu$ for all $\theta$ such that $-\pi/2\leq\mu-\theta<\pi/2$.  Thus $F$ is convex, and hence starlike. Also, \eqref{p5eq23a} shows that the normalization of $f_k$ implies the normalization of $\tilde{f}$. The result now follows by Theorem \ref{p5theom21b}.
\end{proof}
 Using Remark \ref{p5remak7ab}, Theorem \ref{p5theom11} gives  the following result.
\begin{theorem}\label{p5corl8a}
For $k=1,2$, let $f_k=h_k+\overline{g_k}\in\mathcal{S}_H$  such that
\begin{equation}\label{p5eq8b}
h_k(z)-\emph{e}^{2 i \mu}g_k(z)=\int_0^z \psi_{\mu,\nu}(\xi)p(\xi)d\xi, \quad \mu\in[0, \pi), \nu\in[0, 2\pi),
\end{equation}
 where $\psi_{\mu, \nu}$ is given by \eqref{np5eq6} and $p$ is an analytic mapping with $\RE p>0$ on $\mathbb{D}$.  If
\begin{itemize}
\item[(i)] $0\leq \eta\leq1$, or
\item[(ii)]  the dilatation $\omega_k$ of $f_k$ satisfies $|\omega_k(z)|<\alpha_k\leq 1$ and $\eta\in\mathbb{C}$ such that \[ |\eta|\leq\frac{(1-\alpha_1)(1-\alpha_2)}{2(\alpha_1+\alpha_2)},\]
\end{itemize} then the mapping $\tilde{f}$ given by \eqref{p5eq5} is univalent and convex in the direction $\mu$.
\end{theorem}
\begin{proof}
Since $\RE p>0$ on $\mathbb{D}$, we have
\begin{align*}
\RE\bigg(\frac{1}{\psi_{\mu,\nu}(z)}\left(\int_0^z \psi_{\mu,\nu}(\xi)p(\xi)d\xi\right)'\bigg)
&=\RE\bigg(\frac{1}{\psi_{\mu,\nu}(z)}\psi_{\mu,\nu}(z)p(z)\bigg)=\RE p(z)>0.
\end{align*}
 Therefore, by Remark \ref{p5remak7ab}, the mapping $\int_0^z \psi_{\mu,\nu}(\xi)p(\xi)d\xi$ is convex in the direction $\mu$. Hence, in view of equation \eqref{p5eq8b}, Theorem \ref{p5theom11} follows the result.
\end{proof}
\begin{corollary}\label{p5corl8c}
Let $\nu_1$, $\nu_2\in[0, 2\pi)$, $\mu\in[0, \pi)$ and $A,B\geq 0$ with $A+B>0$. Also, for $k=1,2$, let  $f_k=h_k+\overline{g_k}\in\mathcal{S}_H$ such that
\begin{equation}\label{p5eq8d}
h_k(z)+\emph{e}^{2 i \mu}g_k(z)=A\frac{z(1-z e^{ -i \mu}\cos\nu_1)}{1-z^2 e^{-2 i \mu}}+B\int_0^z\psi_{\mu,\nu_2}(\xi)d\xi,
\end{equation}
 where $\psi_{\mu, \nu_2}$ is defined in \eqref{np5eq6}. Then the mapping $\tilde{f}$  given by \eqref{p5eq5} is  univalent and convex in the direction $\mu+\pi/2$ for all $\eta$ given as in Theorem \ref{p5corl8a}.
 \end{corollary}
\begin{proof}
  We can write \eqref{p5eq8d} as \begin{align*}
 h_k(z)+\emph{e}^{2 i \mu}g_k(z)&=\int_0^z \bigg(A\frac{1-2\xi e^{- i \mu}\cos\nu_1+\xi^2 e^{-2 i \mu}}{(1-\xi^2 e^{-2 i \mu})^2}+B\psi_{\mu,\nu_2}\bigg)d\xi\notag\\&= \int_0^z \frac{q(\xi)d\xi}{1-\xi^2 e^{-2 i \mu}}=\int_0^z q(\xi). \psi_{\mu+\pi/2,0}(\xi)d\xi,\notag
 \end{align*}
 where $q$ is given by
 $$q(z)=A\frac{1-2z e^{- i \mu}\cos\nu_1+z^2 e^{-2 i \mu}}{1-z^2 e^{-2 i \mu}}+B\frac{1-z^2 e^{-2 i \mu}}{1-2z e^{- i \mu}\cos\nu_2+z^2 e^{-2 i \mu}}.$$
 Now, for $\gamma\in [0, 2\pi)$,
\begin{align}
\RE\left(\frac{1-z^2 e^{-2 i \mu}}{1-2z e^{- i \mu}\cos\gamma+z^2 e^{-2 i \mu}}\right)\notag&=\frac{1-|z|^4-2\cos\gamma(1-|z|^2)\RE( e^{- i \mu}z)}{|1-2z e^{- i \mu}\cos\gamma+z^2 e^{-2 i \mu}|^2}\notag\\&\geq\frac{(1-|z|^2)(1+|z|^2-2|\cos\gamma|\RE( e^{- i \mu}z))}{|1-2z e^{ -i \mu}\cos\gamma+z^2 e^{-2 i \mu}|^2}>0\notag,\quad z\in\mathbb{D}.
\end{align}
Therefore $\RE q>0$ on  $\mathbb{D}$. The proof now follows by Theorem \ref{p5corl8a}.
 \end{proof}

 \begin{remark}  Corollary \ref{p5corl8c} reduces to
  \cite[Theorem 3]{wang} of Wang \textit{et al.} when $A=1$, $B=0$, $\mu=\pi$ and $\gamma_1=0$ and to
  \cite[Theorem 2.1]{kumar} of Kumar \textit{et al.} when $A=1$, $B=0$ and $\mu=\pi$.
\end{remark}


\begin{thebibliography}{00}
\bibitem{muhana} Y. Abu-Muhanna \ and \ G. Schober. Harmonic mappings onto convex domains. Canad. J. Math. 39 (1987), no. 6, 1489--1530.

\bibitem{ahuja} O. P. Ahuja, Use of theory of conformal mappings in harmonic univalent mappings with directional convexity, Bull. Malays. Math. Sci. Soc. (2) {\bf 35} (2012), no.~3, 775--784.

\bibitem{subzar3} S. Beig, On convolution of harmonic mappings, Complex Anal. Oper. Theory {\bf 14} (2020), no.~4, paper no. 48, 10 pp.

\bibitem{subzar}S. Beig\ and\ V. Ravichandran, Directional convexity of harmonic mappings. Bull. Malays. Math. Sci. Soc. {\bf 41} (2018), 1045--1060.

\bibitem{subzar1} S. Beig\ and\ V. Ravichandran, Convolution and convex combination of harmonic mappings, Bull. Iranian Math. Soc. {\bf 45} (2019), no.~5, 1467--1486.

\bibitem{subzar2}S. Beig, Y. J. Sim\ and\ N. E. Cho, On convex combinations of harmonic mappings, J. Inequal. Appl. {\bf 2020}, Paper No. 84, 14 pp.

\bibitem{boyd} Z. Boyd, M. Dorff, M. Nowak, M. Romney \ and \ M. Woloszkiewicz,  Univalency of convolutions of harmonic mappings. Appl. Math. Comput. {\bf234} (2014), 326--332.

\bibitem{cluine}J. Clunie, \ and \ T. Sheil-Small,  Harmonic univalent functions, Ann. Acad. Sci. Fenn. Ser. A I Math. {\bf 9} (1984), 3--25.

\bibitem{dorff1} M. Dorff, Harmonic univalent mappings onto asymmetric vertical strips. Computational methods and function theory 1997 (Nicosia), 171--175, Ser. Approx. Decompos., 11, World Sci. Publ., River Edge, NJ, 1999.


\bibitem{dorff} M. Dorff, \ and \ S. Rolf,   Anamorphosis, mapping problems, and harmonic univalent functions. Explorations in complex analysis, 197--269, Classr. Res. Mater. Ser., Math. Assoc. America, Washington, DC, 2012.

\bibitem{nowak} M. Dorff, M. Nowak,\ and\  M. Woloszkiewicz, Convolutions of harmonic convex mappings, Complex Var. Elliptic Equ. {\bf 57} (2012), no.~5, 489--503.

\bibitem{her}\'{A}. Ferrada-Salas, R. Hern\'{a}ndez\ and\ M. J. Mart\'{i}n, On convex combinations of convex harmonic mappings, Bull. Aust. Math. Soc. {\bf 96} (2017), no.~2, 256--262.

\bibitem{kumar} R. Kumar, S. Gupta \ and \ S. Singh, Linear combinations of univalent harmonic mappings convex in the direction of the imaginary axis. Bull. Malays. Math. Sci. Soc. {\bf 39} (2016), no. 2, 751--763.

\bibitem{lewy} H. Lewy, On the non-vanishing of the Jacobian in certain one-to-one mappings, Bull. Amer. Math. Soc. {\bf 42} (1936), no.~10, 689--692.


\bibitem{royster}W. C. Royster\ and\ M. Ziegler, Univalent functions convex in one direction. Publ. Math. Debrecen {\bf 23} (1976), no. 3-4, 339--345.

\bibitem{sheil} T. Sheil-Small, Constants for planar harmonic mappings. J. London Math. Soc. {\bf 42} (1990), 237–-248.

\bibitem{shi} L. Shi, Z.-G. Wang, A. Rasila \ and \ Y. Sun, Convex combinations of harmonic shears of slit mappings Bull. Iranian. Math. Sci. Soc. {\bf 43} (2017), 1495--1510.



\bibitem{sun} Y. Sun, A. Rasila\ and \ Y.-P. Jiang, Linear combinations of harmonic quasiconformal mappings convex in one direction. Kodai Math. J. {\bf 39} (2016), no. 2, 366--377.



\bibitem{wang} Z.-G Wang, Z.-H Liu \ and \  Y.-C Li, On the linear combinations of harmonic univalent mappings. J. Math. Anal. Appl. {\bf 400} (2013), no. 2, 452--459.
\end{thebibliography}
\end{document}